\documentclass[11pt]{amsart}

\title[Westervelt equation]{Weakly nonlinear geometric optics for the  Westervelt equation and recovery of the nonlinearity 
}

\usepackage{appendix}

\usepackage{hhline}
\usepackage{cancel} %
\usepackage{amsmath}
\usepackage{marginnote}  %
\usepackage{amssymb}
\usepackage{amsthm}
\usepackage{accents}
\usepackage{amsfonts}
\usepackage{mathtools}
\usepackage{enumerate}                 
\usepackage{comment}
\usepackage{hyperref}
\usepackage{multirow}
\usepackage[noadjust]{cite}
\usepackage{color,soul}
\allowdisplaybreaks

\let \i\undefined

\let \Re \undefined
 \DeclareMathOperator{\Re}{Re}
 \let \Im \undefined
 \DeclareMathOperator{\Im}{Im}

\DeclareMathOperator{\supp}{supp}

\DeclareMathOperator{\SO}{SO}

\numberwithin{equation}{section}%

\newcommand{\i}{\mathrm{i}}

\newcommand{\R}{\mathbb{R}}

\renewcommand{\S}{\mathbb{S}}

\newcommand{\p}{\partial}

\renewcommand{\a}{\alpha}

\let\epsilon\varepsilon

\newcommand{\vertiii}[1]{{\left\vert\kern-0.25ex\left\vert\kern-0.25ex\left\vert #1 
    \right\vert\kern-0.25ex\right\vert\kern-0.25ex\right\vert}}

\newtheorem{theorem}{Theorem}

\newtheorem{proposition}{Proposition}

\theoremstyle{definition}
\newtheorem{remark}{Remark}%

\numberwithin{equation}{section}

\let\implies\Rightarrow

\newcommand{\be}[1]{\begin{equation}\label{#1}}
\newcommand{\ee}{\end{equation}}
\renewcommand{\r}[1]{\eqref{#1}}

\author[N. Eptaminitakis] {Nikolas Eptaminitakis}
\address{Department of Mathematics, Purdue University, West Lafayette, IN 47907}

\author[P. Stefanov]{Plamen Stefanov}
\address{Department of Mathematics, Purdue University, West Lafayette, IN 47907}
\thanks{P.S. partly supported by  NSF  Grant DMS-2154489}

\begin{document}

\date{\today}

\begin{abstract}
We study the non-diffusive Westervelt equation in the weakly nonlinear regime. We show that the leading profile equation is of Burgers' type. We show that a compactly supported nonlinearity $\alpha$ can be reconstructed from the tilt of the transmitted high frequency wave packets sent from different directions since those tilts are proportional to the X-ray transform of $\alpha$. 
\end{abstract} 
\maketitle

\begin{center}

\end{center}

\section{Introduction}   
We study the non-diffusive Westervelt equation
\begin{equation}\label{W1}
\partial_t^2 p-\Delta p -\alpha \partial_t^2p^2=0,\qquad t\in\R,\quad x\in\R^n,
\end{equation}
modeling the evolution of the pressure $p(t,x)$ relative to an equilibrium position, in a non-diffusive medium. Here, $0\le \alpha\in C_0^\infty(\R^n)$, 
with $\supp \alpha \subset B(0,R)$ for some $R>0$ is the nonlinearity coefficient. 
 
We  probe the medium with high frequency single phase waves which propagate into the future in direction $\omega\in \mathbb{S}^{n-1}$, sending them from outside the ball $B(0,R)$ and measuring the transmitted wave outside that ball again. This corresponds to a choice of a phase function $\phi: = -t+x\cdot\omega$ below.  
Assuming that such a wave is created by a source $\delta'(t) p_0(x)$, this leads to Cauchy data $(p,p_t) = (p_0,0)$ at $t=0$, see \r{W1a} below. 
Such data create, on the level of the geometric optics, two waves: one propagating in the direction $\omega$ and another one in the direction $-\omega$, with equal amplitudes (in the linear region). We chose $\omega$ so that it points into $B(0,R)$ when issued from  $\supp\chi$, see \r{W1a}, 
so the  high frequency part of the second wave  only will enter the ball, see Figure~\ref{fig:setup}. 

If $|p|\ll1$, then \r{W1} acts essentially as a linear equation, and the nonlinear term can be ignored. If $|p|$ is not small enough, \r{W1} would not be even hyperbolic. Indeed, expanding the $\partial_t^2p^2$ term, we get
\[
(1-2\alpha p)\partial_t^2 p-\Delta p -2\alpha (\partial_tp)^2=0.
\]
Roughly speaking, the speed is $c_\alpha = (1-2\alpha p)^{-1/2}$, assuming $2\alpha p<1$, at least. Of particular interest is the so-called \textit{weakly nonlinear} geometric optics regime for quasilinear (and also, semilinear) PDEs, characterized by the requirement  that the eikonal equation governing the geometry is unaffected by the nonlinearity while the leading amplitudes (profiles) are affected. This is the regime, as we will see below, leading to the predicted and observed effects in nonlinear ultrasound. Its study started in the physics literature and was developed in the math one in \cite{Metivier-Notes,Metivier-Joly-Rauch, Joly-Rauch_just, Donnat-Rauch_dispersive, Dumas_Nonlinear-Geom-Optics, JMR-95, Rauch-geometric-optics}, and other works for first order symmetrizeable hyperbolic systems of PDEs.  The idea is that the amplitude of the wave should be related to is wavelength in a particular way, note the prefactor $h$ in \r{W1a} below. It turns out that the right scaling for the effects described is when the amplitude is of the same order of magnitude as the wavelength. Since those two quantities are expressed in different units, this has to be understood relative to the space and time scale; 
and also, in an asymptotic sense, as those two quantities both tend to zero. 

We analyze the Westervelt equation in the weakly nonlinear regime and show that the leading profile equation is of Burgers' type. In the applied literature, a reduction to Burgers' equation for a mono phase wave is well-known and derived by ignoring some ``small'' quantities based on a priori but not quite explicit assumptions, see also section~\ref{sec_1D}. A very non-trivial step is to show that there is an exact solution  close to the approximate one. We use a result by Gu\`es \cite{Gues}, and for this purpose we reduce \r{W1} to a first order system. Note that the traditional existence and uniqueness theorems for nonlinear PDEs are not useful here since they require smallness in some high $H^s$ norm, which we do not have. 

We also study the inverse problem of recovery of $\alpha$ from measurements of the wave packets described above once they exit $\supp\alpha$. We show that the tilt on the wave is proportional to the X-ray transform of $\alpha$ in the direction $\omega$, which allows us to recover $\alpha$. When $\alpha$ is small, linearizing the second harmonic leads to the X-ray transform of $\alpha$, as well. 
The inverse problem for the Westervelt equation has been studied earlier in \cite{acosta2021nonlinear} and in \cite{UhlmannZhangNonlinearAcoustics} using the higher order linearization method (the latter dealt with more general hyperbolic quasilinear equations including the Westervelt equation as a special case) and in \cite{Kaltenbacher_2021} using linearization. Those works consider small, relative to ours, signals where the effect of $\alpha$ is pushed to lower order terms. The higher order linearization  method was pioneered in \cite{KLU-18} and \cite{LassasUW_2016}, and used later in 
\cite{Hintz1,Hintz2,LassasUW_2016, LUW1, lassas2020uniqueness, Hintz-U-19, uhlmann-zhang-2021inverse, OSSU-principal}. On the other hand, inverse problems for semilinear wave type PDEs  were studied by the second author and S\'a Barreto in \cite{S-Antonio-nonlinear, S-Antonio-nonlinear2} in regimes where the nonlinearity affects the principal term in \cite{S-Antonio-nonlinear} and the subprincipal one in \cite{S-Antonio-nonlinear2}, because of the nature of the inverse problem there. 

Finally, we want to mention that nonlinear effects are already present in ultrasound imaging of the human body \cite{Wells_1999, humphrey2003non}.   The nonlinearity can create problems if the imaging method is based on a linear model, or it can be useful by creating higher order harmonics, most importantly the second harmonic, which provides additional imaging opportunities.

\section{Main results}  \label{sec_2}
Consider \r{W1} with  Cauchy data at   $t=0$ of the kind
\begin{equation}\label{W1a}
  p  |_{t=0}
  = 2h \chi(x)\cos\frac{x\cdot\omega}{h},\quad p_t |_{t=0}=0,
\end{equation}
where $h>0$ is a small parameter, equal to the wavelength over $2\pi$, and $\chi\in C_0^\infty(\R^n)$, depending on $\omega\in \mathbb{S}^{n-1}$,  with $\supp \chi\cap \overline{B(0,R)}=\emptyset$. 
Because of that, the probing wave $p$ satisfies \eqref{W1} with $\alpha =0$ for $|t|$ small, i.e. initially the wave is not affected by the nonlinearity, see  Figure~\ref{fig:setup}.  
\begin{figure}[b] %
  \centering
   \includegraphics[ %
   scale=0.35,page=1]{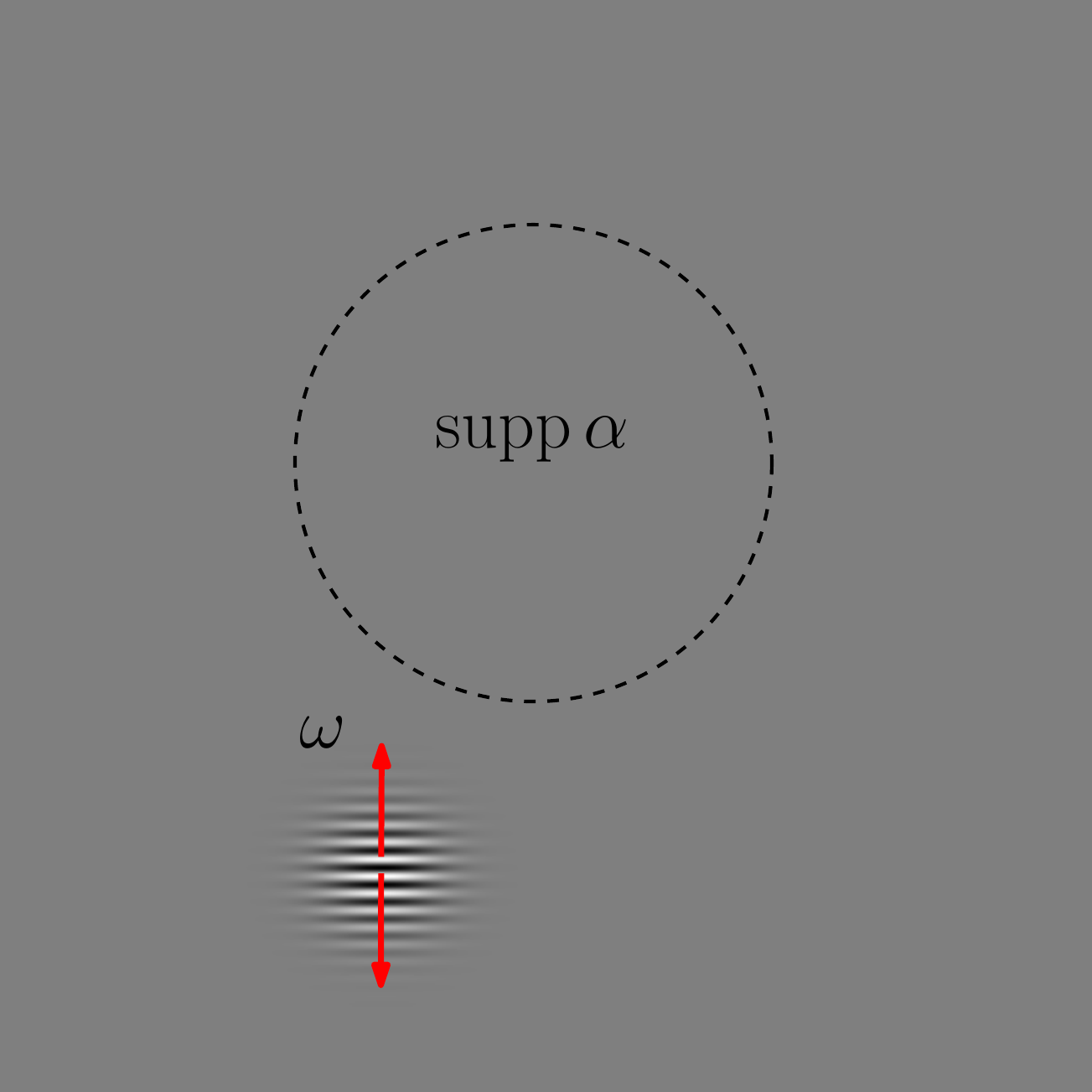}   \hspace{0.4in} 
   \includegraphics[ %
   scale=0.35,page=2]{Westervelt-fig4.pdf}\hspace{0.4in} 
      \includegraphics[ %
   scale=0.35,page=3]{Westervelt-fig4.pdf}
  \caption{The setup. Left: at $t=0$. Center: at $t=t_0>0$. Right: at the terminal time $t=T$. The amplitudes of each wave packet in the center and on the right are equal to $1/2$ of that of the initial one on the left.}
  \label{fig:setup}
\end{figure}
As explained in the introduction, we chose $\chi$ and $\omega$ so that rays from $\supp\chi$ in the direction $\omega$ may hit $\overline{B(0,R)}$ but rays in the opposite direction do not.  

For the purpose of obtaining the Radon transform of $\alpha$ along all rays through $\supp\alpha$, we would need a family of such $\chi$, and a uniform estimate below with respect to that family. With the notation $x= (x',x^n)$, choose $\chi_0\in C_0^\infty(\R^n)$ with $\supp\chi_0\subset B(z_0,R_0)$, where $z_0 = (0,-R-R_0)$ with some $R_0>0$. This function satisfies the requirements for $\omega_0=(0,\dots,0,1)$  (in view of our dimension restrictions,   either $\omega_0= (0,1)$ or $\omega_0= (0,0,1)$), and the rays through $\supp\chi$ in the direction $\omega_0$ include $x'=0$ when $\chi(0)\not=0$.

To generate the family of $\chi$ we need, we introduce two parameters: a shift parameter $y'$ with $|y'|\le R$ to shift $z_0$ to $z_0+(y',0)$ in the disk $x^n=-R-R_0$, $|x'|\le R$, and an  orthogonal matrix $B\in \SO(n)$ to rotate the configuration. In other words, %
choosing $ x = B(\tilde x+(y',0))$, $ \omega = B\omega_0$, 
we have $\tilde x= B' x - (y',0)$, so choosing $\eta=(y',B)$ as a parameter, we take $\chi_\eta(x) =\chi_0(B' x - (y',0)) $, $ \omega = B\omega_0$. 

In the weakly nonlinear regime, one seeks solutions $p$ with a leading term $hU_0(t,x,\phi/h)$, where $U_0$ is $2\pi$-periodic in $\theta$. Our first result states that such solutions, with initial conditions \r{W1a}, exist.

\begin{theorem}\label{thm1}
Let $n=2$ or $n=3$. With notation as above, assume 
\be{thm1_eq1}
(\max|\chi |) \max_x\int\alpha(x+\sigma\omega)d\sigma <1.
\ee
Then for every $T>0$,  $b_0\in (0,1)$, there exists $h_{T,\chi,b_0}>0$ so that \r{W1}, \r{W1a} has a unique solution $p$ satisfying $1-2\a p\geq b_0>0$  for $0\le t\le T$ and $0<h<h_{T,\chi,b_0}$. Moreover, $p=p^-+p^+$ mod $O(h^\infty)$, where $p^-$ is an asymptotic solution of the linear equation ($\alpha=0)$ propagating in the direction $-\omega$, while the part $p^+$ of the solution propagating in the direction of $\omega$ satisfies  
\be{thm1_eq2}
p^+(t,x,h) =hU_0(t,x,\phi/h)+ h^2R(t,x,h),\quad 0\le t\le T, 
\ee
where $\phi=-t+x\cdot\omega$,  $U_0(t,x,\theta)$ is $2\pi$-periodic,  solves  Burgers' equation
\begin{equation}\label{W_B}
(\partial_t+ \omega\cdot\nabla_x) U_0  + \alpha U_0\partial_\theta U_0 =0, \quad  U_0(0,x,\theta) =\chi(x)\cos\theta,
\end{equation}
and $R\in \mathcal{B}_\sigma^3$ for some $\sigma>0$, see \eqref{bmr}.

If $\chi=\chi_\eta$ as above, and in particular $\omega=B'\omega_0$, then the corresponding $h_{T,\chi_\eta,b_0}$ can be taken uniform in $\eta$. 
\end{theorem}

Our next theorem addresses the inverse problem. 

\begin{theorem}\label{thm2} Let $n=2$ or $n=3$. 
Let $T>0$ be such that all the rays through $\supp\alpha$ at $t=0$ in the direction $\omega$ do not meet  $\overline{B(0,R)}$  anymore for $t\ge T$. Assume \r{thm1_eq1}. Then $U_0|_{t=T}$,  and therefore, $p(T,x;\omega)$, known for $x\not\in  \overline{B(0,R)}$ and $0< h\ll1$ determines the line integrals of $\alpha$ along the rays in the direction of $\omega$ passing through $\supp\chi$, uniquely. In particular, varying $\chi$ and $\omega$, we can determine $\alpha$ uniquely from the knowledge of the principal term in \r{thm1_eq2} for $0<h\ll1$. 
\end{theorem}

\section{Geometric optics, profile equations} \label{sec:GO}
\subsection{The linear case}\label{sec:linear}
We start with $\alpha=0$ since the solution $p$ is expected to propagate with unit speed, and therefore, to stay linear for $|t|\ll1$. In fact the latter can be proved using the Duhamel's principle treating the nonlinearity as a source, assuming existence of a solution. 

The geometric optics construction in the linear case is well known. We will replace the initial condition $2\chi(x)\cos(x\cdot\omega/h)$ in \r{W1a} by $2\chi(x)\exp(\i\, x\cdot\omega/h)$ (we can cancel the factor $h$ in the linear case), and take the real part of the solution later. Note that we cannot do this for the nonlinear equation. 
Looking for a solution of the kind %
$p=e^{\i \phi/h}a$, $a=a_0+ha_1+\dots$, we get that there are two phase functions equal to $x\cdot\omega$ for $t=0$: $\phi_\pm : = \mp t+x\cdot\omega$. We used the notation $\phi$ for $\phi_+$ above, and will return to it later. Then the ansatz is
\be{go0}
p = e^{\i \phi_+/h}a_++ e^{\i \phi_-/h}a_-,
\ee
where $a_\pm \sim a_{\pm,0} + h a_{\pm,1}+\dots$. 
This implies the initial conditions
\[
a_++a_-=2\chi, \quad a_--a_+=0, \quad \text{at $t=0$},
\]
therefore,
\be{go1}
a_+ = a_-=\chi(x),\quad \text{at $t=0$}.
\ee
Conjugating the wave operator $\Box$ with the two exponentials, we get
\[
e^{-\i \phi_\pm/h}\Box e^{\i \phi_\pm/h} = 2\i h^{-1}(\mp \partial_t - \omega\cdot\partial_x) + \Box.
\]
The transport equations are then derived from 
\[
2\i h^{-1}(\mp \partial_t - \omega\cdot\partial_x)a_\pm + \Box a_\pm=0.
\]
We get
\begin{equation}\label{g01aa}
(\pm \partial_t + \omega\cdot\partial_x)a_{\pm,0}=0, \quad (\pm \partial_t + \omega\cdot\partial_x)a_{\pm,j}=\frac1{2\i} \Box a_{\pm,j-1}, \quad j\ge1.
\end{equation}
Then by \r{go1},
\be{go1a}
a_{\pm,0} = \chi(x\mp t \omega),
\ee
and the transport equations for $a_{\pm,j}$, $j\ge1$, can be solved by integration using the zero initial conditions implied by \r{go1}. Indeed, we can rotate the coordinates to assume $\omega = (0,\dots,0,1)$. We will use the notation $x=(x',x^n)$. Then $\pm \partial_t + \omega\cdot\partial_x = \pm \partial_t + \partial_{x^n}$. Pass to characteristic variables $(t,x)\mapsto (t,x',y^n=x^n \mp t)$; then one integrates in $t$. 
For example, one can compute 
\be{go1b}
a_{+,1}= \frac{\i t}2 \Delta_{x'}\chi(x',x^n-t),
\ee
with a similar formula for $a_{-,1}$. All coefficients $a_{+,j}$ with $j\ge1$ are sums of terms involving positive powers of $\Delta_{x'}$ applied to $\chi(x',x^n-t)$ (and $x^n$ derivatives and powers of $t$).  In particular, if we choose $\chi$ so that $\chi(x',x^n)=1$ on some open set of $x'$, say in the ``center of the beam'', then all those higher order coefficients would vanish there, and the principal ones would be just $1$. The coefficients corresponding to $j$ odd are pure imaginary, and the ones corresponding to $j$ even are real, assuming $\chi$ real. 

Fix $0<t_0\ll1$ so that the support of the so-constructed solution is still outside $\overline{B(0,R)}$ for that time, see Figure~\ref{fig:setup}. Set 
\be{go2}
\begin{split}
p_0^+ &= \Re (e^{\i \phi_+/h} ha_+)|_{t=t_0} = h\big(\cos(\phi_+/h)\Re a_+ - \sin(\phi_+/h)\Im a_+\big)\big |_{t=t_0}, \\
 q_0^+ &= \Re \partial_t (e^{\i \phi_+/h} ha_+)|_{t=t_0},
\end{split}
\ee
see \r{go0}.  
 We will use $p_0^+$ as an initial condition at $t=t_0$ for the solution in the nonlinear region, and we will verify that its time derivative coincides with $q_0^+$. Note that the real Fourier expansion of $p_0^+$ in the $\theta=\phi_+/h$ variable (see next section) contains $\cos\theta$ and $\sin\theta$ only.

\subsection{The nonlinear case} \label{sec_nonl}

We want to extend the construction in the previous section  to allow the ``$+$'' term in \r{go0} to enter the ball $B(0,R)$, where the nonlinearity $\alpha$ is supported. The term corresponding to $\phi_-$ there will never hit $B(0,R)$, so it will stay linear. 

\subsubsection{The first profile equation} 
We are looking for a solution of \r{W1}, modulo an $O(h^\infty)$ error, of the form
\begin{equation}\label{W2}
p\sim h U(t,x,\phi/h)
  \sim hU_0 (t,x,\phi /h) + h^2 U_1 (t,x,\phi  /h)+\dots,
\ee
where the profiles $U_j(t,x,\theta)$ are $2\pi$-periodic in $\theta$. Here, $\phi= -t+x\cdot\omega$, which we called $\phi_+$ above. We are not including $\phi_-$ because we want this solution to be an extension of the linear solution corresponding to $\phi_+$ (only), and for $|t-t_0|\ll1$, it will be equal to it, actually. We are taking the initial condition
\be{W2b}
hU|_{t=t_0} =  p_0^+, 
\ee
see \r{go2}. This implies the following 
\begin{equation}\label{W2a}
   U_0(t_0,x,\theta) =\chi(x-t_0\omega)\cos\theta,
 \end{equation} 
see \r{go1a}.  
Plug \r{W2} into \r{W1} and compare the equal powers of $h$. The eikonal equation stays the same, and $\phi$ solves it. 
The next term in the expansion yields the following transport equation
\begin{equation}\label{W3}
(\partial_t+ \omega\cdot\nabla_x) \partial_\theta U_0 + \alpha  \partial_\theta ( U_0 \partial_\theta U_0)=0. 
\end{equation}
Integrating in $\theta$, we get
\begin{equation}\label{W4}
(\partial_t+ \omega\cdot\nabla_x) U_0  + \alpha U_0\partial_\theta U_0 =\beta(t,x),
\end{equation}
with a yet to be determined $\beta$. Since we are constructing an ansatz here, we can make assumptions, and we assume $\beta=0$. In fact, we will show below that to have a $2\pi$-periodic solution of the lower order profile equation, we must have $\beta=0$. 
Make the change of variables $(s,y) = (t-t_0,x-t \omega)$. 
Then $(t,x) = (s+t_0, y+(s +t_0 )\omega)$, and $\partial_s = \partial_t+\omega\cdot \nabla_x$. 
Passing to the variables $(s,y)$, we get 
\begin{equation}\label{W4a}
\partial_s U_0 +  \alpha U_0 \partial_\theta U_0=0, \quad U_0|_{s=0}= \chi(y\cdot\omega) \cos\theta. 
\end{equation}
This is Burgers' equation on a cylinder, i.e., the spatial variable $\theta$ is $2\pi$ periodic. 

\subsubsection{Analysis of the Burgers' equation} \label{sec_Burgers}
Thinking of $y$ and $\omega$ as parameters, equation \r{W4a} takes the form
\begin{equation}\label{W6}
\p_s U_0  + \alpha U_0\partial_\theta U_0 =0, \quad U_0|_{s= 0} = M\cos\theta. 
\end{equation}
where $M=\chi(y\cdot\omega)$, and $\alpha = \alpha(y+s\omega)$. We think of $\alpha$ as a function of $s$. The method of characteristics says that we have to solve
\begin{equation}\label{W6b}
\frac{d}{d s} \theta = \alpha u, \quad \frac{d}{d s}u=0
\end{equation}
with initial conditions $u=M\cos q$, $\theta=q$ for $s=0$. Then $u=M\cos q$ for all $s$, and $\theta = q+ u\int_0^s\alpha(\sigma)\, d\sigma$. Therefore, $U_0$, up to the shock formation, is an implicit solution of the equation
\begin{equation}\label{implicit}
U_0 = M\cos\Big(\theta -  U_0\int_0^s\alpha(\sigma)\, d\sigma\Big). 
\end{equation}
To determine the shock time, make the change of variables $\tilde s=\int_0^s\alpha(\sigma)\, d\sigma$ (assuming for a moment $\alpha>0$ all the time). Then \r{W6} transforms into 
\be{W6c}
\p_{\tilde s} U_0  +   U_0\partial_\theta U_0 =0, \quad U_0|_{\tilde s= 0} = M\cos\theta.
\ee
Then it is well known that first shock develops when $\tilde s=1/M$. Therefore, the shock time for $s$ is the solution of
\[
\int_0^s\alpha(\sigma)\, d\sigma= \frac1M. 
\]
The minimal one corresponds to $M=\max\chi$. In particular, if 
\be{W-cond}
 \int\alpha(\sigma)\,d\sigma< \frac1{\max\chi},
\ee
then there is no shock, and $U_0$ exists for all $s$. 

It is easy to see now that the condition $\alpha>0$ can be removed. With $U_0$ defined as the solution of \r{W6c}, as a function of $\tilde s$, and $\theta$; setting $\tilde s=\tilde s(s)$ produces a solution (which we know is unique on the interval of existence) of \r{W6}.

Next, $U_0$ is periodic in $\theta$ by \eqref{implicit}. Indeed, set 
\[
F(t,x,\theta,u)=u-M\cos\Big(\theta-u\int_0^s\alpha(\sigma)\,d\sigma\Big).
\]
Then $U_0$ is defined as the unique solution of $F=0$, up to the shock time, and since $F$ is $2\pi$-periodic in $\theta$, so  is $U_0$.

\subsubsection{The second profile term} 
The next term in the expansion yields
\begin{equation}\label{W5}
(\partial_t+ \omega\cdot\nabla_x)\partial_\theta U_1  + \alpha\partial_\theta ^2( U_0U_1)= 
\frac12 \Box U_0 + \alpha\partial_t\partial_\theta U_0^2, \quad U_1|_{t=t_0}= \i a_{+,1}\sin\theta,
\end{equation}
see \r{go1b}, and note that $\i a_{+,1}$ is real. 
Integrating in $\theta$, one obtains a linear PDE for $U_1$ with a source term.

We look at the ``DC'' component $U_0^0$ of $U_0$, i.e. the zeroth Fourier mode, which is also the mean value in $\theta$, rescaled. If we subtract it from $U_0$, then $\beta=0$ in \r{W4}. In \r{W5}, the only term with possibly non-trivial zeroth Fourier component would be $\Box U_0^0$. So we get $\Box U_0^0=0$,  with $U_0^0=0$ for $t=t_0$ by the initial condition \eqref{W2a} for $U_0$ (no zeroth component there). Also,   we have  $\partial _t U_0^0 \big|_{t=t_0}=0$ by the second equation in \eqref{go2}.
Thus $U_0^0=0$ and $\beta=0$. 

We will now check, using the fact that $\beta=0$, that the solutions to \eqref{W5} are  periodic for $t\in [t_0,T]$, where $T$ is such that the leading order profile $U_0$ is smooth for $t$ in an open interval $I\supset[t_0,T]$.
First integrate \eqref{W5} with respect to $\theta$ to get 
  \begin{equation}\label{tr}
    (\partial_t+\omega\cdot \nabla_x +\alpha U_0 \partial_\theta)U_1+(\alpha \partial_\theta U_0) U_1=\frac12\int_0^\theta \Box U_0 (\sigma)d\sigma+\alpha \partial_t U_0^2+\beta_1(t,x),
  \end{equation}
  where we wrote $U_0(\sigma)=U_0(t,x,\sigma)$.
Since $U_0$ is periodic, $X:=\partial_t+\omega\cdot \nabla_x +\alpha U_0 \partial_\theta$ is a smooth vector field on the cylinder $I\times \R^n\times \S^1$.
We can take $t$ as a parameter for its integral curves, which are then defined for $t\in I$.
Moreover, $\alpha \partial_\theta U_0 $ and $\alpha \partial_t U_0^2+\beta_1(x,t) $ are smooth $2\pi$-periodic functions with respect to $\theta$, and
 $\int_0^{\theta}\Box U_0(\sigma)d\sigma$ is periodic because it is the integral of a periodic function which has zero mean over the interval $[0,2\pi)$.
So \eqref{tr} is a transport equation on the cylinder, with the initial data (see \eqref{W5}) being smooth and compactly supported on the embedded hypersurface $S=\{t_0\}\times \R^n\times \S^1$, to which $X$ is nowhere tangent.
Then by \cite[Theorem 9.51]{LeeSmooth} for instance,   we conclude that there exists a neighborhood of $S$ in the cylinder (which is constructed by following the integral curves of $X$ for $t\in I$, so it can be taken to be all of $I\times \R^n\times \S^1$)  on which there exists a unique smooth solution to \eqref{tr};  thus the solution $U_1$ is $2\pi$-periodic in $\theta$.

\section{Well posedness}

In this section we show that near the asymptotic solution we constructed there exists an actual solution of \eqref{W1}, using a result of Gu\`es (\cite{Gues}), which we now state for the reader's convenience. Note that the classical well-posedness results for nonlinear PDEs with small data 
do not work here since the $H^s$ norms of the initial conditions are not small when $s\ge2$. 

Consider a system of PDEs written as
\begin{equation*}\label{gues_system}
  L(a(y),u(y))u(y)=F(a(y),u(y)),\quad y=(t,x)\in \R^{n+1},
\end{equation*}
where $u:\R^{n+1}\to \R^N$ is the unknown function, $a:\R^{n+1}\to \R^{N'}$ is a given  function, and the operator $L$ is of the form 
\begin{equation*}
  L(a,u):=\partial_t+\sum_{j=1}^{n}A_j(a,u)\partial_{x^j},\qquad A_j\in C^\infty(\R^{N'+N}; \R^{N\times N} ).
\end{equation*}
The presence of $a(y)$, rather than just setting $a(y)=y$, allows for dependence on parameters; this is useful since in Theorem~\ref{thm:gues} below, $a$ belongs to a class of functions with the estimate uniform in that class. 
It assumed that $L$ is hyperbolic symmetrizeable, in the sense that there exists a positive definite symmetric matrix $S\in C^\infty(\R^{N+N'};\R^{N\times N})$ such that $SA_j$ is symmetric for all $j$.
It is also assumed that    $F\in C^\infty(\R^{N'+N};\R^N)$ satisfies $F(0,0)=0$.

The following  spaces of functions depending on a small parameter $h$ will be needed.
For $T>0$ fixed, $\rho >0$, and $m\in \mathbb{Z}_{\geq 0}$, write
\begin{equation}\label{amr}
\begin{split}
  \mathcal{A}_\rho^m& =  \Big\{u_h\in C^0([0,T];W^{m,\infty}(\mathbb{R}^n)):\; \| u_h(t)\|_{L^\infty(\mathbb{R}^n)}\leq \rho, \text{and}
  \\
  &\qquad \| (h \partial_x)^\alpha u_h(t)\|_{L^\infty(\mathbb{R}^n)}\leq \rho { h} \text{ for } 1\leq |\alpha|\leq m, \;h\in (0,1], \;t\in [0,T]\Big\} ,
\end{split}
\end{equation}
and
\begin{equation}\label{bmr}
\begin{split}
  \mathcal{B}_\rho^m = \Big\{u_h\in C^0([0,T];  H^{m}(\R^n)):\;&  \| (h \partial_x)^\alpha u_h(t)\|_{L^2(\R^n)}\leq \rho \\
  &\text{for } |\alpha|\leq m, \; h\in (0,1],\; t\in [0,T]\Big\}.
\end{split}
\end{equation}
Now for  $\underline{a}_h\in h^M\mathcal{B}_\rho^{m}+\mathcal{A}_{\rho}^{m+1}$, and for $g_h\in h^M\mathcal{B}^m_\rho+\mathcal{A}^{m+1}_{\rho}$  independent of $t$, consider the Cauchy problem 
\begin{equation}\label{Cauchy1}
    \begin{aligned}
  L(\underline{a}_h,\underline{u}_h)\underline{u}_h&=F(\underline{a}_h,\underline{u}_h),\\
  \underline{u}_h(0,x)&=g_h(x),
\end{aligned}    
\end{equation}
for which we seek an exact solution $\underline{u}_h$.

\begin{theorem}[{\cite{Gues}}]\label{thm:gues}
Let $m>n/2+1$ and $M\geq m$. Also let $T$ be fixed. For given $\rho>0$, there exists $h_{\rho,T}>0$ and $\sigma_{\rho,T}>0 $ such that for all $ \underline{a}_h\in h^M\mathcal{B}_\rho^{m}+\mathcal{A}_{\rho}^{m+1}$, if $v_h\in \mathcal{A}_{\rho}^{m+1}$ is an approximate solution of the Cauchy problem \eqref{Cauchy1} in the sense that for some $a_h\in \underline{a}_h+h^M\mathcal{B}_\rho^{m}$, 
 and $r_h\in \mathcal{B}_\rho^m$, it satisfies
\begin{equation*}
  L(a_h,v_h)v_h=F(a_h,v_h)+h^Mr_h,
\end{equation*}
then for all given Cauchy data $g_h\in v_h \big|_{t=0}+h^M\mathcal{B}_\rho^m$, the Cauchy problem \eqref{Cauchy1} admits, for $0<h<h_{\rho,T}$, a unique solution $\underline{u}_h\in v_h+h^M \mathcal{B}^m_\sigma$ on $[0,T]\times \R^n$.
\end{theorem}

\begin{remark}
  The structural hypotheses regarding the operator $L$ mentioned in \cite[Sec.~ 1.2.1]{Gues} are not needed for the proof of Theorem \ref{thm:gues}, but rather for the construction of approximate solutions for general hyperbolic symmetrizeable operators.
\end{remark}

We now turn to our particular case. We assume that the space dimension is $n\leq 3$ (in  higher dimensions we would need to produce more terms for the asymptotic expansion of the solutions of \eqref{W1}).
Our goal is to show that the hypotheses of Theorem \ref{thm:gues} are satisfied, for $m=M=3$.
In order to prove the last statement of Theorem \ref{thm1}, we would like to allow the incoming position and  direction of the probing wave to vary, which we model with the parameter $\eta$, see Section~\ref{sec_2}. 
It turns out that it is more convenient to  keep them fixed and apply the rigid motion parameterized by $\eta$ to the nonlinearity, by setting $\alpha_\eta(x) =\alpha(B(x+(y',0)))$ and keeping $\omega=\omega_0$ and $\chi=\chi_0$ fixed. Then we are solving
\begin{equation}\label{rotated}
\begin{aligned}
    \partial_t^2 p-\Delta p -\alpha_\eta(x) \partial_t^2p^2=0,\\  
  p  |_{t=0}
  = 2h \chi(x)\cos\frac{x\cdot\omega}{h},\quad p_t |_{t=0}=0.
\end{aligned}
\end{equation}

We need to convert \eqref{rotated} into an equivalent 1st order system.
Set
\begin{equation*}\label{psi}
\psi = \int_{0}^t p(\tau,x)\, d\tau, \quad v = \nabla \psi. 
\end{equation*}
Notice that $\psi \big|_{t=0}=0$ and so $v \big|_{t=0}=0$. 
Since $\nabla\cdot v=\Delta \psi$, $\psi_t=p$, $\p_t p \big|_{t=0}=0$,
 taking the $t$-antiderivative of the PDE in \eqref{rotated} yields
\begin{equation}\label{system}
  \begin{split}
(1-2\alpha_\eta p)\p_t p &= \nabla\cdot v \\
\p_t v &= \nabla p,
\end{split}
\end{equation}
 and the initial condition 
\begin{equation}\label{IC}
  (p,v)\big|_{t=0}=\Big(2h\chi(x)\cos\frac{x\cdot \omega}{h},0\Big).
\end{equation}
Note that $v$ also satisfies the additional condition  
\begin{equation*}\label{S0}
\nabla\times v=0
\end{equation*}
when $n=3$, interpreted as $\nabla \times (v_1,v_2,0)=0 $ when $n=2$. It is preserved along the flow and satisfied by  the initial conditions, so it is actually redundant.
The process above shows that a solution of \eqref{rotated} satisfies \eqref{system}, \eqref{IC}. 
Conversely, given a solution of \r{system},   \r{IC}, we get that $p$ solves the PDE \r{rotated}, with $p=2h\chi(x)\cos({x\cdot \omega}/{h})$ for $t=0$. About $\p_tp|_{t=0}$, we get from \r{system} and \r{IC} that for $t=0$ we have  $(1-2\alpha_\eta p)\p_tp=\nabla\cdot v=0$, hence $\p_tp=0$, for the small solutions we are interested in. Therefore, $p$ solves \r{rotated}.

We are interested in solutions of \eqref{system}, \eqref{IC} which are small in the $L^\infty$ sense, and in particular ones which satisfy $1-2\a_\eta p\geq b_0>0$, for some fixed $b_0\in (0,1)$.
Taking $\psi\in C^\infty(\R;[0, 1])$ with $ \psi\equiv 0$ on $[1,\infty)$ and $\psi \equiv 1$ on $(-\infty, 1-b_0]$, we observe that the solutions of \eqref{system} that satisfy $1-2\a_\eta p\geq b_0$ for all $t\in [0,T]$ and $x\in \R$ also solve \eqref{system} with $(1-2\alpha_\eta p)$ replaced by $(1-(2\alpha_\eta p)\,\psi(2\alpha_\eta p))$, and vice versa.
Now for such $\psi$ write
\begin{equation}\label{S1}
  \partial_t u+\sum_{j=1}^n S_{\psi}(\alpha_\eta(x),u)^{-1}A_j\p_{x_j}u=0, \qquad u \big|_{t=0}=(p,v)^{\mathrm{T}} \big|_{t=0},
\end{equation}
where $u=(p,v)^{\rm{T}}$, the entries of $A_j\in \R^{(n+1)\times (n+1)}$ are given by $(a^j)_{kl}=-(\delta_{0k}\delta_{jl}+\delta_{l0}\delta_{kj })$ for $j=1,\dots,n$, and 
$S_{\psi}(a,u)=\mathrm{diag}(1-2ap\,\psi(2a p),1,\dots,1)\in C^\infty(\R\times \R^{n+1};\R^{(n+1)\times (n+1)})$ is symmetric and positive definite by the choice of $\psi$.
The quasilinear system \eqref{S1} is then hyperbolic symmetrizeable, and all of its sufficiently small solutions agree with those of \eqref{rotated}.

%

%
%
%
%
%

%

\smallskip

Recall now the profiles produced in Section \ref{sec:GO} (with $\alpha$ replaced by $\alpha_\eta$ everywhere).  
Let $T>0$, $h_0>0$ be  such that smooth solutions to \eqref{W4}, \eqref{W2a} (with $\beta=0$)  and of \eqref{W5}   exist on $[t_0,T]\times \mathbb{R}^n$ 
for $h\in (0,h_0]$ and such that for $t\in [0,T]$ we have 
$|2\alpha_\eta(x)p_{*}^+(t,x)|\ll1-b_0,$ 
where
\begin{equation}\label{S00}
  p_{*}^+(t,x)=\begin{cases}
    \Re(e^{\i\phi_+/h}(ha_{+,0}+h^{2}a_{+,1})), & 0\leq t <t_0,\\
    h U_0\big(t,x,\frac{-t+x\cdot \omega}{h}\big)+h^2 U_1\big(t,x,\frac{-t+x\cdot \omega}{h}\big), & t_0\leq t\leq T.
  \end{cases}      
\end{equation}
with the $a_{+,j}$ as in Section \ref{sec:linear}. For the backward approximate solution, write
\begin{equation}\label{S01}
  p_{*}^-(t,x)=\Re(e^{\i\phi_-/h}(ha_{-,0}+h^{2}a_{-,1})),\quad 0\leq t\leq T.
\end{equation}
Note that the finite speed of propagation for solutions of the linear transport equations \eqref{g01aa} and \eqref{W5} and of Burgers' equation \eqref{W4}, together with the compact support of their initial data, implies that $p_{*}^\pm$ are compactly supported in $x$ for each $t\in [0,T].$
Also notice that the choice of initial conditions for $U_j$, $j=0,1$, implies that $p_{*}^+$ is continuous in $t$, with values in $H^m(\R^n)$ or in $W^{m,\infty}(\R^n)$, for any $m$: indeed, the only possible discontinuity is at $t_0$.
To see that there is no discontinuity, note that for $j=0,1$ and for $t\leq t_0$, $\Re(e^{\i \theta}a_{+,j})$ satisfies the same (linear) transport equation with smooth coefficients that $U_j$ does for $0\leq t-t_0\ll 1$,  and their data match at $t=t_0$.

We first check that \eqref{S00} is an approximate solution of \eqref{rotated}.
The PDEs \eqref{g01aa} for $a_{\pm, j}$, $j=0,1$, (resp. \eqref{W4a}, \eqref{W5} for $U_j$) were  produced by inserting  the ansatz \eqref{go0} multiplied by $h$ (resp. ansatz \eqref{W2}) into the Westervelt equation \eqref{rotated} and matching coefficients at orders $h^{-1}$, $h^0$ and $h^1$. (The coefficient at order $h^{-1}$ cancels automatically due to the a priori choice of the phase; if we had non-constant speed of propagation for the linear equation, that term would correspond to the eikonal equation.)
Therefore the result of inserting \eqref{S00} into \eqref{rotated} in the linear case consists of the contributions of the    terms of \eqref{rotated} at order $h^2$, and in the nonlinear case at order $\geq h^2$ up to order $h^{4}$.
In the nonlinear case, the contributions at orders $h^2$ to $h^4$ consist of finite sums of at most two-fold products of
  $U_0$, $U_1$ and their derivatives.
The  backward propagating approximate  wave \eqref{S01} does not interact with the nonlinearity and propagates  for time $t\in[0,T]$, so it satisfies \eqref{rotated} with $\alpha_\eta=0$ up to a compactly supported $O(h^2)$ term.
So, upon setting $p_*=p_{*}^-+p_{*}^+$ we have
\begin{equation}\label{error1}
  \partial_t^2 p_{*}-\Delta p_{*} -\alpha_\eta\partial _t^2 p_{*}^2=h^2\sum_{k=0}^{2}h^k P_{k}\big(t,x,\frac{-t+x\cdot \omega}{h}\big)+h^2 Q\big(t,x,\frac{t+x\cdot \omega}{h}\big)
\end{equation}
with $P_{k}$, $Q$ smooth and  compactly supported in $x$ for $t\in [0,T]$.
So there exists $\rho>0$ such that for $h<h_0$
\begin{equation}\label{error2}
\partial_t^2 p_{*}-\Delta p_{*} -\alpha_\eta\partial _t^2 p_{*}^2=:R_h^+\big(t,x,\frac{-t+x\cdot \omega}{h}\big)+R_h^-\big(t,x,\frac{t+x\cdot \omega}{h}\big)\in h^2\mathcal{B}_\rho^4.
\end{equation}

We can now use our approximate solutions for \eqref{rotated} to produce ones for the equivalent first order system \eqref{system}. Consider
\begin{equation}\label{approx}
  u_*=(p_{*},v_*)=\Big(p_{*}(t,x),\int_{0}^{t}\nabla_x p_{*}(\tau,x)d\tau\Big).
\end{equation}
We first check that $ u_*\in \mathcal{A}_\rho^4$.
By \eqref{S00}-\eqref{S01} it follows that $p_{*}\in \mathcal {A}_\rho^4$ for some $\rho>0$.
Moreover, for $H\in C_0^\infty(\R\times \R^n \times \mathbb{S}^1)$ write
\begin{equation}\label{ibp}
\begin{split}
  \int_{0}^t \nabla_x \Big(H\Big(\tau,x, &\frac{-\tau+x\cdot \omega}{h}\Big)  \Big)d\tau\\
  &=
  \int_{0}^t (\nabla_x H)\Big(\tau,x,\frac{-\tau+x\cdot \omega}{h}\Big)+\frac{\omega}{h} \partial_ \theta H\Big(\tau,x,\frac{-\tau+x\cdot \omega}{h}\Big) d\tau\\
&=\int_{0}^t (\nabla_x H)\Big(\tau,x,\frac{-\tau+x\cdot \omega}{h}\Big)-\omega \partial_\tau \Big( H\Big(\tau,x,\frac{-\tau+x\cdot \omega}{h} \Big)\Big)\\
&{}\qquad+\omega (\partial_ t H) \Big(\tau,x,\frac{-\tau+x\cdot \omega}{h}\Big) d\tau\\
&=\int_{0}^t (\nabla_x H)\Big(\tau,x,\frac{-\tau+x\cdot \omega}{h}\Big)+\omega  (\partial_ t H) \Big(\tau,x,\frac{-\tau+x\cdot \omega}{h}\Big) \Big)d\tau\\
& \qquad-\omega \Big(H\Big(t,x,\frac{-t +x\cdot \omega}{h}\Big)-H\Big(0,x,\frac{x\cdot \omega}{h}\Big)\Big).
  \end{split}
\end{equation}
Taking 
$$H=\begin{cases}
  \Re(e^{\i \phi_+/h}a_{+,j}), &0\leq t< t_0\\
  U_j, & t_0\leq t\leq T
\end{cases}, \quad j=0,1,$$ 
and combining with \eqref{S00}, the computation above demonstrates that $\int_{0}^{t}\nabla_x p_{*}^+(\tau,x)d\tau\in \mathcal{A}_\rho^4$ for some $\rho$.
Similarly,
$\int_{0}^{t}\nabla_x p_{*}^-(\tau,x)d\tau\in \mathcal{A}_\rho^4$.
Thus $\int_{0}^{t}\nabla_x p_{*}(\tau,x)d\tau\in \mathcal{A}_\rho^4$.

We now check  that $u_*$ is an approximate solution to \eqref{system}.
Indeed, this follows by inserting \eqref{approx} into \eqref{system} and using integration and \eqref{error1}--\eqref{error2} to see that $u_*$ is an approximate solution of \eqref{system} up to an error of the form $(\int_{0}^t\big(R_h^++R_h^-\big) d\tau,0)$.
An argument using the fundamental theorem of calculus  similar to \eqref{ibp} then shows that $\int_{0}^t\big( R_h^++R_h^-\big)d\tau\in h^3\mathcal{B}_\rho^3$ for some $\rho>0$.
Since $1-2\alpha_\eta p_*=1-2\a p_*\,\psi(2\a p_*)$ for small $h$, we observe that $u_*$ is also an approximate solution of \eqref{S1} up to an error in $h^3\mathcal{B}_\rho^3$.

Now we have the following. 
\begin{proposition}\label{approx_sol}
Fix $T>0$ such that the profile equations  \eqref{W4} and \eqref{W5} (with $\alpha$ replaced by $\alpha_\eta$) admit smooth solutions for all $t\in [t_0,T]$, for all $\eta$. By shrinking  $h$, for a given $b_0\in (0,1)$ assume that $|2\alpha_\eta p_*|\ll 1-b_0$ on $[0,T]$, for all $\eta$. Let $\rho>0$  be large enough to ensure that $\alpha_\eta\in \mathcal{A}_\rho^4$  and that for all $\eta$, we have that  $u_*\in \mathcal{A}_\rho^4$ and  it is an approximate solution of the  system \eqref{system}, \eqref{IC} up to an error in $h^3\mathcal{B}_\rho^3$.
Then for such $\rho $ there exist $\sigma_{T,\rho,b_0}$ and $h_{T,\rho,b_0}$  such that for any $\eta$ and any Cauchy data in the space $u_* \big|_{t=0}+h^3\mathcal{B}_\rho^3$ there exists for $h\in (0,h_{T,\rho,b_0}]$ a unique solution $\tilde{u}_*=(\tilde{p}_*,\tilde{v}_*)\in u_*+h^3\mathcal{B}_\sigma^3$ of \eqref{system} on $[0,T]\times \R^n$ satisfying  $2\alpha_\eta \tilde{p}_*\leq 1-b_0$ there.

\end{proposition}{}
\begin{proof}
An approximate solution $u_*=(p_*,v_*)$ of \eqref{system}, \eqref{IC} for which $p_*$ is  sufficiently small in the $L^\infty $ sense, as is the one in our hypothesis, is also an approximate solution of \eqref{S1}, as already mentioned.
Applying Theorem \ref{thm:gues} for \eqref{S1} with the approximate solution \eqref{approx}, we conclude that for any Cauchy data in $u_* \big|_{t=0}+h^3\mathcal{B}_\rho^3$ there exists a unique solution $\tilde{u}_*^\psi\in u_*+h^3\mathcal{B}_\sigma^3$ of \eqref{S1} on $[0,T]\times \R^n$, for some $\sigma>0$ and for $h$ sufficiently small.
To show that $\tilde{u}_*^\psi=(\tilde{p}^\psi_*,\tilde{v}^\psi_*)$ is also a solution of \eqref{system}, \eqref{IC}, it suffices to show that  
\begin{equation}\label{h_limit}
  w_h\in h^3\mathcal{B}_{\sigma}^3\implies \sup_{0\leq t\leq T}\|w_h(t)\|_{L^\infty(\R^n)}=O(h)
\end{equation} uniformly in $w_h$, 
 for then we can guarantee that
\begin{equation}\label{lower_bound}
  2\alpha_\eta p_*^\psi\leq 1-b_0\quad \text{for all } (t,x)\in [0,T]\times \R^n
\end{equation}
by shrinking $h$, since $p_*$ is small.
Moreover, $\tilde{u}_*^\psi $ must be the only solution of \eqref{system}, \eqref{IC} satisfying \eqref{lower_bound}, because any such solution also satisfies \eqref{S1}.
Now \eqref{h_limit} follows from the $\R^n$ version of the Sobolev embedding theorem (see e.g. \cite[Theorem 3.26]{McLean-book}): since $2>n/2$ for $n\leq 3$, there exists a $C$ such that for each $t\in [0,T]$ we have
\begin{equation*}
  \begin{aligned}
     &\| w_h(t)\|_{L^\infty(\R^n)}\leq C \|w_h(t)\|_{H^2(\R^n)}\leq C \sum_{|\alpha|\leq 2}\|\p^\alpha_x w_h(t)\|_{L^2(\R^n)}
  \\
  &\qquad\leq C h^{-2}\sum_{|\alpha|\leq 2}\|(h\p_x)^\alpha w_h(t)\|_{L^2(\R^n)}  \leq h C\Big( h^{-3}\sum_{|\alpha|\leq 3}\|(h\p_x)^\alpha w_h(t)\|_{L^2(\R^n)}\Big)\leq Ch \sigma
  \end{aligned}
\end{equation*}
by our assumption, and the claim is proved.
\end{proof}

\begin{remark}
Varying the incoming direction $\omega$ and position  
in the initial data does not result in a small perturbation of them, in the sense of Proposition \ref{approx_sol}; note that $\omega$ is divided by $h$ in \eqref{IC}.
This is the reason why we translated and rotated the  nonlinearity instead.  
\end{remark}

Note that the statement of the proposition allows us to construct a true solution of \eqref{system} having precisely the Cauchy data \eqref{IC}.

\begin{proof}[Proof of Theorem \ref{thm1}]
Fix any $T>0$, $\omega=\omega_0\in \S^{n-1}$ and $b_0\in (0,1)$. As discussed in Section \ref{sec_nonl},  assumption \eqref{thm1_eq1} implies that the profile equations \eqref{W4} and \eqref{W5} admit smooth solutions for all $t\in [0,T]$, and thus an approximate solution $u_*=(p_*,v_*)$ to the first order system \eqref{system} under the initial condition \eqref{IC} can be constructed for $t\in [0,T]$.
Moreover, $p_*$ is small in the $L^\infty$ sense when $h$ is.
By Proposition \ref{approx_sol}, there exists $h_{T,b_0}$ and $\sigma_{T,b_0}$ such that for each $\eta$,  \eqref{system}, \eqref{IC} has a unique solution $u=(p,v)$ for $0<h\leq h_{T,b_0}$, differing from an approximate one by an element in $h^3\mathcal{B}_\sigma^3$ and satisfying $2\alpha_\eta p\leq 1-b_0$. 
Then, as discussed immediately following \eqref{IC}, $p$ is the unique solution of \eqref{rotated}. %
Equations \eqref{thm1_eq2} and \eqref{W_B} follow by construction of the approximate solution.
\end{proof}

\section{Numerical experiments}
\subsection{One dimension} \label{sec_1D}
Assuming $\alpha$ depending on one variable only, say $x^n$ (then it is not compactly supported but this is not a problem in this case), the problem becomes 1D. We have
\begin{equation*} 
\partial_t^2 p-\p_x^2 p -\alpha \partial_t^2p^2=0
\end{equation*}
with initial conditions
\begin{equation*} 
u|_{t\ll0}= h\chi(-t+x)\cos\frac{-t+x}{h}.
\end{equation*}
In the applied literature, they pass to variables $(\tau,y)= (t-x,x)$. This transforms \r{W7} into
\[
2\p_\tau \p_x p - \p_x^2 p- \alpha \p_\tau^2p^2=0. 
\]
Then they argue that in this moving time frame, the $x$ derivatives are small, so the $\p_x^2 p$ can be ignored. Then they integrate with respect to $\tau$ to get Burgers' equation
\[
\p_x p -  \alpha p \p_\tau p=0. 
\]
This is the transport equation \r{W6} since $\tau=-\phi$ (the variables in \r{W6} are somewhat different though). 

We work in $[-1,1]\ni x$. Since the speed is one, the time needed to cross from the one end to the other of the interval is $2$ but the wave is centered at $x=-0.7$ at $t=0$ and we stop the computations when $t=T:=1.4$. We take $\alpha(x) = e^{-(x/0.3)^2/2}$, $h=0.02$, and $\chi(x) $ is ``essentially supported'' in $[-1, -0.75]$, where $\chi$ is almost zero; with $\chi\ge0$, $\max\chi=1.5$. Condition \r{thm1_eq1} is fulfilled then with the left hand-side being approximately $0.8$. We compute the packet moving to the right only. 
Then it starts from the left and essentially moves to the right slightly changing its shape. The top travels faster, the bottom slower. This is what Burgers' equation predicts. In Figure~\ref{fig:1}, we show the packet at $t=T$, compared with the linear one. The relative shift on the top, for example, is proportional to the integral of $\alpha$ along the way, which gives the Radon transform of it in the multi-dimensional case. 

\begin{figure}[h] %
   \includegraphics[trim={1.3in 0 0 0}, scale=0.35]{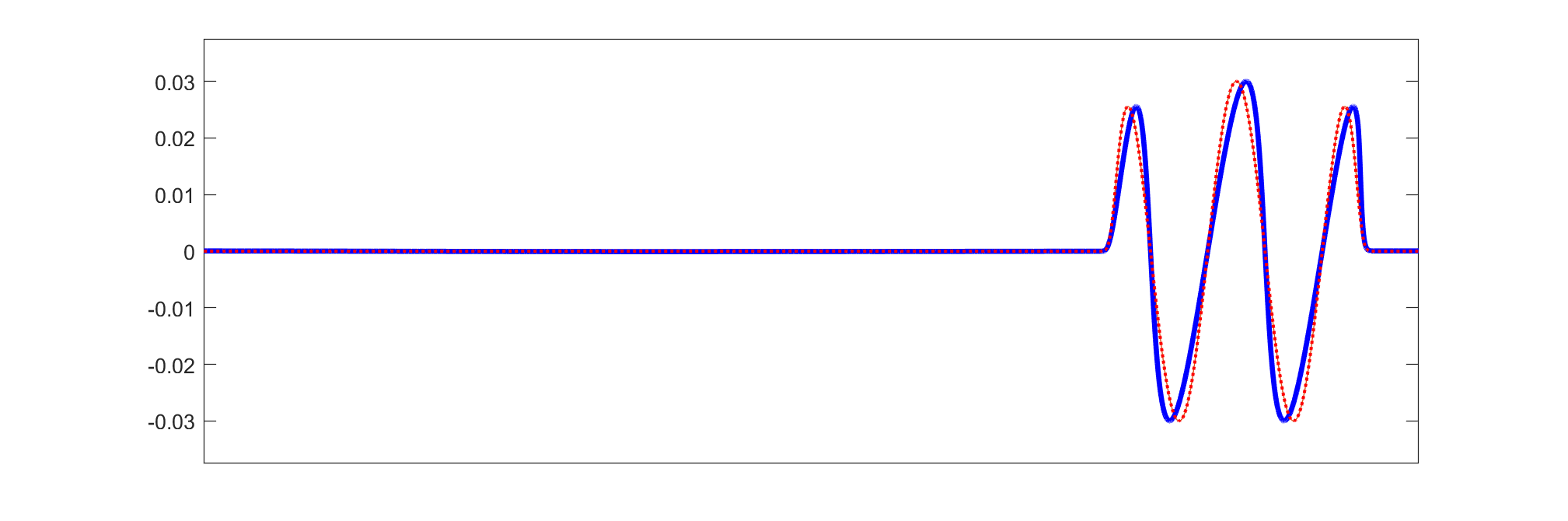}
  \caption{1D: A wave packet at its terminal time $t=T$. It moves to the right. The dotted red curve is the linear solution, just for comparison.}
  \label{fig:1}
\end{figure}

\subsection{A 2D example}

We run a 2D experiment. The initial condition is a wave packet, of the kind shown in Figure~\ref{fig:setup}, moving up. After passing through the nonlinearity, we plot it in Figure~\ref{fig:2} on the right, vs.\ the linear solution on the left. The curving of the wave fronts is natural (for the linear equation, as well), and is due to the  wavelength not being small enough compared to the size of the wave packet, for numerical reasons. The nonlinear solution  has maxima speeding up and minima lagging behind. This effect is stronger in the center where the amplitude is larger. This is due to the Burgers' transport equation. 

\begin{figure}[h] %
  \centering
   \includegraphics[trim={0 1in 1in  0}, scale=0.25]{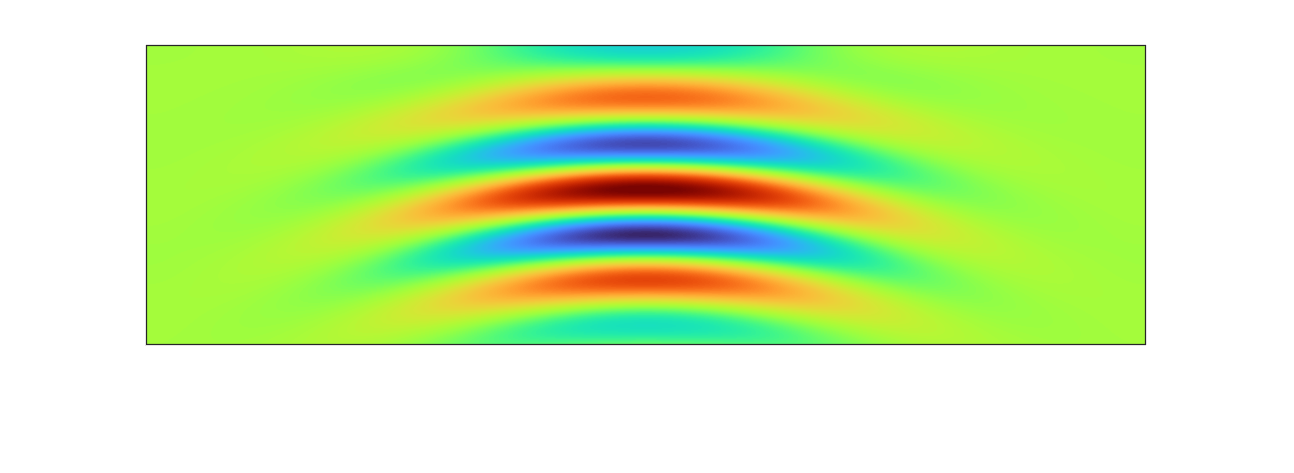}
   \includegraphics[trim={1in  1in 0 0}, scale=0.25]{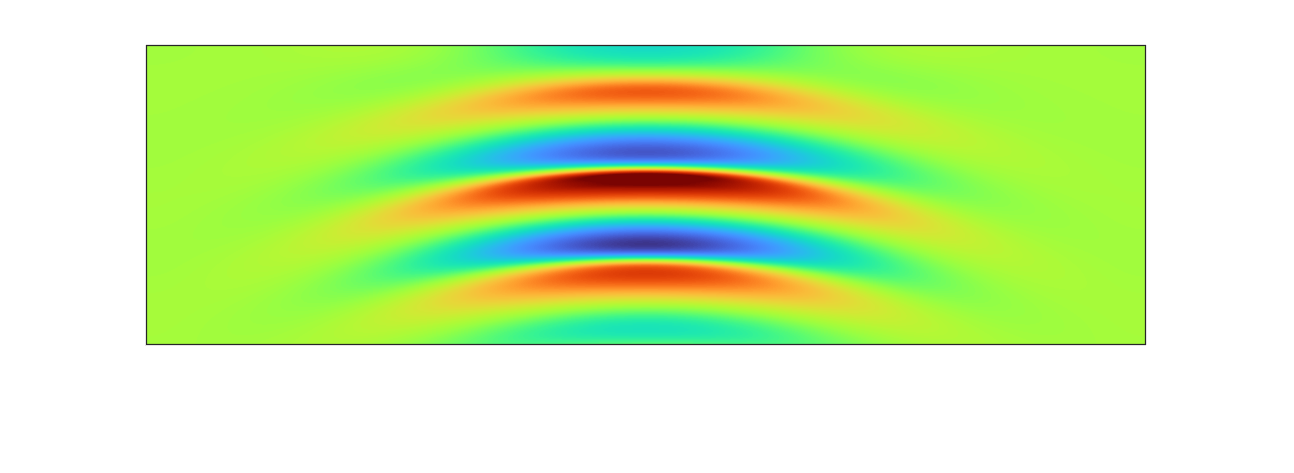}
  \caption{2D: the wave packet at its terminal time. Left: the linear solution. Right: the nonlinear one. The maxima are shifted more upwards, towards the direction of the movement, while the minima are lagging behind.}
  \label{fig:2}
\end{figure}

\section{The Inverse Problem. Proof of Theorem~\ref{thm2}}
\subsection{Proof of Theorem~\ref{thm2}} 
Assume the conditions of Theorem~\ref{thm2} satisfied. 
We send a wave packet in direction $\omega$, fixed,  and our  measurement is  the leading order profile $U_0$, which satisfies \eqref{W6},  after it interacts with the nonlinearity.
We will recover the X-ray transform of $\alpha$ in the direction determined by the vector $\omega$. To do this, we will compare it to the leading order profile $U_0^L$  of the linear solution ($\alpha=0$). For it, we have
\[
U_0^L= \chi({-t+x\cdot \omega})\cos(\theta).
\]
Note first that knowing \r{thm1_eq2} for $0<h\ll1$ allows us to recover $U_0(T,x,\theta)$ for $x\not\in B(0,R)$, and in particular for those $x$ corresponding to endpoints of exiting rays, see also Figure~\ref{fig:setup}. This can be done by recovering the Fourier coefficients of $U_0$ as in  \cite[Proposition~3.2]{S-Antonio-nonlinear2}. 

Now pass to the variables $(t,x)=(s,y+s\omega)$. They are the same as those in Section~\ref{sec_nonl} with $t_0=0$ formally. Then
\[
U_0^L= \chi(y\cdot \omega)\cos(\theta),
\]
while $U_0$ solves \r{W6}. Our data is $U_0(s=T,y,\theta)$ for $y$ such that $\chi(y\cdot \omega)\not=0$, and all $\theta$. Fix such an  $y$.

Recall that \r{W6} is solved by the method of characteristics, see \eqref{W6b}. That method admits the following characterization. In the $(s,y)$ coordinates, each level set $U_0=k$ moves with speed $k\alpha(s)$, staying at the same level,  and the resulting curve is the graph of $U_0$. In particular, the zeros do not move. The points between two zeros can move only within that interval because there is no shock up to $t=T$, and therefore, the characteristics through them cannot intersect those through the zeros, which are vertical lines in the  $(\theta,s)$ plane. With $y$ fixed as above, choose $\theta$ so that $\cos\theta\not=0$. Then $k:= \chi(y\cdot\omega)\cos\theta\not=0$. Along the horizontal line on the graph of $\theta\mapsto U_0^L$, let $d$ be the signed distance to the closest point on the graph of $\theta\mapsto U_0$, which is between the same two consecutive zeros of $\cos\theta$, see Figure~\ref{fig:sines}. Then $d=k\int_0^T\alpha(\sigma) d\sigma$ by the formula following \r{W6b}. Since $k\not=0$ is known, we can recover the integral, and therefore, the X-ray transform of $\alpha$ in the direction $\omega$ along any such ray. 
\begin{figure}[h] %
   \includegraphics[trim={0 0 0 0}, scale=0.5,page=4]{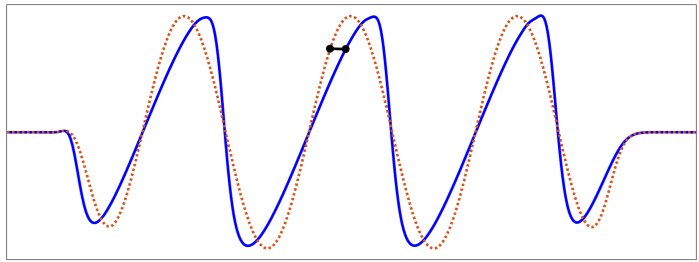}
  \caption{A cross-section along a fixed ray parallel to $\omega$ at $t=T$ compared to the linear solution, dotted. Fixing a level $p=k$, the horizontal shift along it is proportional to $k\int\alpha$ with the integral taken along the ray.}
  \label{fig:sines}
\end{figure}

This concludes the proof of Theorem~\ref{thm2}.

\subsection{Fourier decomposition and higher order harmonics}
When $\alpha\ll1$, the tilt can be  too small to be measured reliably. Then one can measure the second harmonic. 
As we found out in Section~\ref{sec_Burgers}, $U_0$ has a vanishing zeroth harmonic. Then in the coordinates $(s,y)$, 
\[
U_0(s,\theta) = \sum_{k\in\mathbb{Z}\setminus 0} e^{\i k\theta} u_k(s),
\]
where $u_k$ depend on $y$ through $M=\chi(y\cdot\omega)$, $\omega$ fixed. 
Then $u_k$ must solve, see \r{W6},
\be{W7}
\frac{d}{d s} u_k +\frac{\i}2 k \alpha(s)\sum_{k_1+k_2=k}u_{k_1}u_{k_2} =0
\ee
with initial conditions
\be{W8}
u_{1}=u_{-1}=\frac{M}2, \quad u_j=0, \; j\not=\pm1, \quad \text{at $s=0$}.  
\ee
The sum there is a discrete convolution, $u_k*u_k$. Using the same arguments as in \cite[p.38]{S-Antonio-nonlinear}, one can prove that $u_{-k} = \bar u_k$. 

As already mentioned above, see \r{W6c}, it is enough to solve this problem for $\alpha=1$, call the solution $\tilde u_k$; then $u_k(s)= \tilde u_k(\tilde s(s))$ with $\tilde s=\int_0^s\alpha(\sigma)\, d\sigma$ solves \r{W7}, and by uniqueness, it is the only solution up to the shock formation. 

First, $\tilde u_k|_{\tilde s=0}$ are as in \r{W8}. 
 Since $\partial_s = \partial_{\tilde s}$ at $s=0$, we have
\[
\partial_{\tilde s} \tilde u_2 = -\partial_{\tilde s} \tilde u_{-2} = -\i\frac{M^2}4, \quad  \partial_{\tilde s} \tilde u_k=0 \text{ for $k\not=\pm2$}, \quad \text{at $\tilde s=0$}.  
\]
Therefore, the Taylor expansion of $\tilde u$ at $\tilde s=0$ (which is actually analytic up to the shock time) is
\[
\tilde U_0 = M\cos\theta+ \frac{M^2}2\tilde s\sin(2\theta)+O(\tilde s^2).
\]
Note that one can get the same result by  Burgers' equation directly. Passing to the solution $u$, and considering $\alpha$ as a small parameter, we get
\[
 U_0 = \chi(y\cdot\omega)\cos\theta+ \frac12 \chi^2(y\cdot\omega) \sin(2\theta) \int_0^s\alpha(\sigma)\, d\sigma+O(\alpha^2).
\]
The second term is the second harmonic, up to $O(\alpha^2)$. Actually, one can show in the same way that the error above is $O(\alpha^3)$. Therefore, the second harmonic of $U_0$ at $t=T$, linearized near $\alpha=0$,  recovers the X-ray transform of $\alpha$. This has been used in ultrasound: one filters the data to recover the second harmonic only, which is less likely to backscatter right away and create a very strong signal, thus allowing for better sensitivity deeper into the tissue \cite{humphrey2003non}.

We mention here that time-periodic solutions to the diffusive Westervelt equation have been studied in \cite{Kaltenbacher-periodic}.
There, similarly to our analysis, terms are created at harmonic  frequencies which are integer multiples of the excitation frequencies.

%

\begin{thebibliography}{LLPMT20}

\bibitem[AUZ21]{acosta2021nonlinear}
Sebastian Acosta, Gunther Uhlmann, and Jian Zhai.
\newblock Nonlinear ultrasound imaging.
\newblock {\em arXiv:2105.05423}, 2021.

\bibitem[DR97]{Donnat-Rauch_dispersive}
Phillipe Donnat and Jeffrey Rauch.
\newblock Dispersive nonlinear geometric optics.
\newblock {\em J. Math. Phys.}, 38(3):1484--1523, 1997.

\bibitem[Dum06]{Dumas_Nonlinear-Geom-Optics}
Eric Dumas.
\newblock About nonlinear geometric optics.
\newblock {\em Bol. Soc. Esp. Mat. Apl. SeMA}, (35):7--41, 2006.

\bibitem[Gu{\`{e}}93]{Gues}
Olivier Gu{\`{e}}s.
\newblock D{\'{e}}veloppement asymptotique de solutions exactes de syst{\`eme}s
  hyperboliques quasilin{\'{e}}aires.
\newblock {\em Asymptotic Anal.}, 6(3):241--269, 1993.

\bibitem[HU19]{Hintz-U-19}
Peter Hintz and Gunther Uhlmann.
\newblock Reconstruction of {L}orentzian manifolds from boundary light
  observation sets.
\newblock {\em Int. Math. Res. Not. IMRN}, (22):6949--6987, 2019.

\bibitem[Hum03]{humphrey2003non}
V.~F. Humphrey.
\newblock Non-linear propagation for medical imaging.
\newblock {\em WCU}, 2003:73--80, 2003.

\bibitem[HUZ21a]{Hintz2}
Peter Hintz, Gunther Uhlmann, and Jian Zhai.
\newblock {An Inverse Boundary Value Problem for a Semilinear Wave Equation on
  {L}orentzian Manifolds}.
\newblock {\em International Mathematics Research Notices}, 05 2021.

\bibitem[HUZ21b]{Hintz1}
Peter Hintz, Gunther Uhlmann, and Jian Zhai.
\newblock The {D}irichlet-to-{N}eumann map for a semilinear wave equation on
  {L}orentzian manifolds.
\newblock {\em arxiv:2103.08110}, 2021.

\bibitem[JMR95]{JMR-95}
J.-L. Joly, G.~M\'{e}tivier, and J.~Rauch.
\newblock Coherent and focusing multidimensional nonlinear geometric optics.
\newblock {\em Ann. Sci. \'{E}cole Norm. Sup. (4)}, 28(1):51--113, 1995.

\bibitem[JR92]{Joly-Rauch_just}
Jean-Luc Joly and Jeffrey Rauch.
\newblock Justification of multidimensional single phase semilinear geometric
  optics.
\newblock {\em Trans. Amer. Math. Soc.}, 330(2):599--623, 1992.

\bibitem[Kal21]{Kaltenbacher-periodic}
Barbara Kaltenbacher.
\newblock Periodic solutions and multiharmonic expansions for the {W}estervelt
  equation.
\newblock {\em Evol. Equ. Control Theory}, 10(2):229--247, 2021.

\bibitem[KLU18]{KLU-18}
Yaroslav Kurylev, Matti Lassas, and Gunther Uhlmann.
\newblock Inverse problems for {L}orentzian manifolds and non-linear hyperbolic
  equations.
\newblock {\em Invent. Math.}, 212(3):781--857, 2018.

\bibitem[KR22]{Kaltenbacher_2021}
Barbara Kaltenbacher and William Rundell.
\newblock Determining the nonlinearity in an acoustic wave equation.
\newblock {\em Math. Methods Appl. Sci.}, 45(7):3554--3573, 2022.

\bibitem[Lee03]{LeeSmooth}
John~M. Lee.
\newblock {\em Introduction to smooth manifolds}, volume 218 of {\em Graduate
  Texts in Mathematics}.
\newblock Springer-Verlag, New York, 2003.

\bibitem[LLPMT20]{lassas2020uniqueness}
Matti Lassas, Tony Liimatainen, Leyter Potenciano-Machado, and Teemu Tyni.
\newblock Uniqueness and stability of an inverse problem for a semi-linear wave
  equation.
\newblock {\em arXiv preprint arXiv:2006.13193}, 2020.

\bibitem[LUW17]{LUW1}
Matti Lassas, Gunther Uhlmann, and Yiran Wang.
\newblock Determination of vacuum space-times from the {E}instein-{M}axwell
  equations.
\newblock {\em arXiv:1703.10704}, 2017.

\bibitem[LUW18]{LassasUW_2016}
Matti Lassas, Gunther Uhlmann, and Yiran Wang.
\newblock Inverse problems for semilinear wave equations on {L}orentzian
  manifolds.
\newblock {\em Comm. Math. Phys.}, 360(2):555--609, 2018.

\bibitem[M\'09]{Metivier-Notes}
Guy M\'{e}tivier.
\newblock The mathematics of nonlinear optics.
\newblock In {\em Handbook of differential equations: evolutionary equations.
  {V}ol. {V}}, Handb. Differ. Equ., pages 169--313. Elsevier/North-Holland,
  Amsterdam, 2009.

\bibitem[McL00]{McLean-book}
William McLean.
\newblock {\em Strongly elliptic systems and boundary integral equations}.
\newblock Cambridge University Press, Cambridge, 2000.

\bibitem[MJR99]{Metivier-Joly-Rauch}
Guy M\'{e}tivier, Jean-Luc Joly, and Jeffrey Rauch.
\newblock Recent results in non-linear geometric optics.
\newblock In {\em Hyperbolic problems: theory, numerics, applications, {V}ol.
  {II} ({Z}\"{u}rich, 1998)}, volume 130 of {\em Internat. Ser. Numer. Math.},
  pages 723--736. Birkh\"{a}user, Basel, 1999.

\bibitem[OSSU20]{OSSU-principal}
Lauri Oksanen, Mikko Salo, Plamen Stefanov, and Gunther Uhlmann.
\newblock Inverse problems for real principal type operators.
\newblock {\em arXiv preprint arXiv:2001.07599}, 2020.

\bibitem[Rau12]{Rauch-geometric-optics}
Jeffrey Rauch.
\newblock {\em Hyperbolic partial differential equations and geometric optics},
  volume 133 of {\em Graduate Studies in Mathematics}.
\newblock American Mathematical Society, Providence, RI, 2012.

\bibitem[SBS21a]{S-Antonio-nonlinear}
Ant\^{o}nio S\'{a}~Barreto and Plamen Stefanov.
\newblock Recovery of a cubic non-linearity in the wave equation in the weakly
  non-linear regime.
\newblock {\em arXiv:2102.06323, to appear in Communication in Mathematical
  Physics}, 2021.

\bibitem[SBS21b]{S-Antonio-nonlinear2}
Ant\^{o}nio S\'{a}~Barreto and Plamen Stefanov.
\newblock Recovery of a general nonlinearity in the semilinear wave equation.
\newblock {\em arXiv:2107.08513}, 2021.

\bibitem[UZ21]{uhlmann-zhang-2021inverse}
Gunther Uhlmann and Yang Zhang.
\newblock Inverse boundary value problems for wave equations with quadratic
  nonlinearities.
\newblock {\em arXiv:2104.08386}, 2021.

\bibitem[UZ22]{UhlmannZhangNonlinearAcoustics}
Gunther Uhlmann and Yang Zhang.
\newblock An inverse boundary value problem arising in nonlinear acoustics,
  2022.

\bibitem[Wel99]{Wells_1999}
P~N~T Wells.
\newblock Ultrasonic imaging of the human body.
\newblock {\em Reports on Progress in Physics}, 62(5):671--722, {J}an 1999.

\end{thebibliography}

\end{document}